\documentclass[12pt]{article}
\usepackage{amssymb}
\usepackage{amsthm}
\usepackage{amsmath}
\usepackage{graphicx}
\usepackage{enumerate}
\usepackage{pict2e}
\usepackage{framed}

\DeclareMathOperator{\Con}{Con}

\DeclareMathOperator{\Max}{Max}
\DeclareMathOperator{\Min}{Min}

\newtheorem{theorem}{Theorem}[section]
\newtheorem{definition}[theorem]{Definition}
\newtheorem{lemma}[theorem]{Lemma}
\newtheorem{proposition}[theorem]{Proposition}

\newtheorem{remark}[theorem]{Remark}
\newtheorem{example}[theorem]{Example}
\newtheorem{corollary}[theorem]{Corollary}
\title{Congruences on posets, relatively pseudocomplemented and Boolean posets}
\author{Ivan~Chajda and Helmut~L\"anger$^1$}
\date{}

\begin{document}

\footnotetext[1]{Corresponding author}

\maketitle
	
\begin{abstract}
The aim of the present paper is to extend the concept of a congruence from lattices to posets. We use an approach different from that used by the first author and V.~Sn\'a\v sel in \cite{CS}. By using our definition we show that congruence classes are convex. If the poset in question satisfies the Ascending Chain Condition as well as the Descending Chain Condition, then these classes turn out to be intervals. If the poset has a top element $1$ then the $1$-class of every congruence is a so-called strong filter. We study congruences on relatively pseudocomplemented posets which form a formalization of intuitionistic logic. For such posets we define so-called deductive systems and we show how they are connected with congruence kernels. We prove that every strong filter $F$ of a relatively pseudocomplemented poset induces a congruence having $F$ as its kernel. Finally, we consider Boolean posets which form a natural generalization of Boolean algebras. We show that congruences on Boolean posets in general do not share properties known from Boolean algebras, but congruence kernels of Boolean posets still have some interesting properties.
\end{abstract}
	
	

\section{Introduction}

Congruences form a useful tool for studying the structure of lattices. They enable the construction of a quotient lattice and, moreover, they classify certain classes of lattices. For example, a lattice is distributive if and only if every of its filters is a class of at least one congruence, and if a lattice is relatively complemented then every of its filters is a kernel of at most one congruence. Consequently, every filter of a Boolean algebra is the kernel of exactly one congruence.

Bounded posets form a generalization of lattices. In particular, this analogy becomes apparent when the poset in question satisfies both the Ascending Chain Condition and the Descending Chain Condition (ACC and DCC, for short). In such a case, the sets $\Max L(x,y)$ of all maximal elements below $x$ and $y$, and $\Min U(x,y)$ of all minimal elements over $x$ and $y$ are not empty, consist of mutually incomparable elements and can be used for replacing the operations $\wedge$ (i.e.\ infimum) and $\vee$ (i.e.\ supremum) in lattices, respectively. We can ask if it is possible to generalize the concept of a congruence from lattices to bounded posets in some natural way in order to get results for bounded posets similar to those known for lattices. It is worth noticing that congruences on posets were already investigated by the first author and V.~Sn\'a\v sel in \cite{CS}, however, in a different way. Namely, in \cite{CS} congruences on posets were defined as equivalence relations induced by so-called $LU$-morphisms. In the present paper we use another approach. We define congruences as equivalence relations compatible with the binary operators $\Max L$ and $\Min U$ mentioned above. This means that in case of lattices these congruences coincide with lattice congruences. Moreover, every class of such a congruence is a convex subset and, if the poset in question satisfies both the ACC and the DCC, in particular if it is finite, then every of its classes is an interval, and these classes can be ordered by their upper (or equivalently lower) bounds. Hence, the quotient poset can be embedded into the poset itself.

The second reason of our new definition is that this concept of congruence can be extended to posets with additional operations. In our paper we investigate congruences in relatively pseudocomplemented posets and in Boolean posets. This is motivated by the fact that relatively pseudocomplemented posets serve as an algebraic axiomatization of intuitionistic logic, see e.g.\ \cite B, \cite K, \cite N and \cite{ST}, and Boolean posets are a natural generalization of Boolean algebras that serve as an algebraic axiomatization of classical logic. One can compare our results with those obtained for relatively pseudocomplemented semilattices in \cite{CHK} and \cite{CL01}. For relatively pseudocomplemented posets we define so-called deductive systems. We show that such systems are filters, and every strong filter $F$ induces a congruence whose kernel is just $F$.

A rather different situation appears for Boolean posets. Although Boolean posets are a natural generalization of Boolean algebras, congruences on them do not share known congruence conditions valid for Boolean algebras. Namely, congruences on Boolean posets need neither be permutable, nor regular, nor uniform. Moreover, not every filter of a Boolean poset is the kernel of some congruence. Further, the congruence lattice of a Boolean poset need neither be distributive nor at least modular. However, we prove certain relationships between filters and congruences on Booelan posets. Moreover, every reflexive binary relation on a Boolean poset compatible with the binary operators $\Max L$ and $\Min U$ turns out to be a congruence.

\section{Preliminaries}

Let $\mathbf P=(P,\le)$ be a poset, $a,b\in P$ and $A,B\subseteq P$. We define
\begin{align*}
\Max A & :=\text{set of all maximal elements of }A, \\
\Min A & :=\text{set of all minimal elements of }A, \\
A\le B & \text{ if }x\le y\text{ for all }x\in A\text{ and all }y\in B, \\
  L(A) & :=\{x\in P\mid x\le A\}, \\
  U(A) & :=\{x\in P\mid A\le x\}.
\end{align*}
Here and in the following we identify singletons with their unique element. Instead of $L(\{a\})$, $L(\{a,b\})$, $L(A\cup\{a\})$, $L(A\cup B)$ and $L\big(U(A)\big)$ we simply write $L(a)$, $L(a,b)$, $L(A,a)$, $L(A,B)$ and $LU(A)$. Analogously we proceed in similar cases.

If $(P_1,\le)$ is a set with a binary relation then a bijection $f$ from $P_1$ to $P$ is called an {\em order isomorphism} from $(P_1,\le)$ to $\mathbf P$ if for all $x,y\in P_1$, $x\le y$ is equivalent to $f(x)\le f(y)$. In such a case, $(P_1,\le)$ is a poset, too.

All posets considered in this paper are assumed to be non-empty.

Recall that a poset is said to satisfy the Ascending Chain Condition (ACC, for short) or the Descending Chain Condition (DCC, for short) if it contains no infinite ascending or descending chain, respectively. Of course, every finite poset satisfies both the ACC and the DCC.

Let $P$ be a non-empty set, $n\ge1$, $Q$ an $n$-ary operator on $P$, i.e.\ a mapping from $P^n$ to $2^P$. Moreover, let $R$ be a binary relation on $P$. We call $R$ to be {\em compatible with $Q$} if $(a_1,b_1),\ldots,(a_n,b_n)\in R$ implies that there exists some $a\in Q(a_1,\ldots,a_n)$ and some $b\in Q(b_1,\ldots,b_n)$ with $(a,b)\in R$. Of course, if $^+$ is e.g.\ a unary operation on $P$ then $R$ is compatible with $^+$ if $(a,b)\in R$ implies $(a^+,b^+)\in R$.

On any set $M$, by $\Delta$ and $\nabla$ we denote the smallest and greatest equivalence relation on $M$, respectively, i.e.
\[
\Delta:=\{(x,x)\mid x\in M\}\text{ and }\nabla:=M^2.
\]

\section{Congruences on posets}

\begin{definition}
For a poset $\mathbf P=(P,\le)$ we will call an equivalence relation on $P$ being compatible with the binary operators $\Max L$ and $\Min U$ a {\em congruence} on $\mathbf P$. Let $\Con\mathbf P$ denote the set of all congruences on $\mathbf P$ and for $\Theta\in\Con\mathbf P$ let $P/\Theta$ denote the set of all classes of $\Theta$.
\end{definition}

Obviously, if a binary relation on $P$ is compatible with the binary operators $\Max L$ and $\Min U$ then it is also compatible with the binary operators $L$ and $U$, but not vice versa.

\begin{theorem}\label{th1}
Let $\mathbf P=(P,\le)$ be a poset, $\Theta,\Phi\in\Con\mathbf P$, $(a,b)\in\Theta$ and $c,d\in P$. Then the following holds:
\begin{enumerate}[{\rm(i)}]
\item There exists some $c\in\Max L(a,b)$ and some $d\in\Min U(a,b)$ with $c,d\in[a]\Theta$,
\item $[a]\Theta$ is convex,
\item $(c,d)\in\Theta$ if and only if there exists some $(e,f)\in\Theta$ with $c,d\in[e,f]$,
\item $\Theta=\Phi$ if and only if $\Theta\cap\{(x,y)\in P^2\mid x\le y\}=\Phi\cap\{(x,y)\in P^2\mid x\le y\}$.
\end{enumerate}
\end{theorem}

\begin{proof}
\
\begin{enumerate}[(i)]
\item Since $\Theta$ is compatible with the binary operators $\Max L$ and $\Min U$ we have
\begin{align*}
\big(a\times\Max L(a,b)\big)\cap\Theta & =\big(\Max L(a,a)\times\Max L(a,b)\big)\cap\Theta\ne\emptyset, \\
\big(a\times\Min U(a,b)\big)\cap\Theta & =\big(\Min U(a,a)\times\Min 
U(a,b)\big)\cap\Theta\ne\emptyset
\end{align*}
and hence there exists some
\[
(c,d)\in\Max L(a,b)\times\Min U(a,b)
\]
with $(a,c),(a,d)\in\Theta$. This shows
\[
(c,d)\in\big(\Max L(a,b)\times\Min U(a,b)\big)\cap([a]\Theta)^2
\]
and hence
\[
\big(\Max L(a,b)\times\Min U(a,b)\big)\cap([a]\Theta)^2\ne\emptyset.
\]
\item Assume $a\le b$ and $e\in[a,b]$. Since $\Theta$ is compatible with the binary operator $\Max L$ we have
\[
\{(a,e)\}\cap\Theta=\big(\Max L(a,e)\times\Max L(b,e)\big)\cap\Theta\ne\emptyset
\]
and hence $e\in[a]\Theta$ showing convexity of $[a]\Theta$.
\item This follows from (i) and (ii).
\item If $\Theta\cap\{(x,y)\in P^2\mid x\le y\}=\Phi\cap\{(x,y)\in P^2\mid x\le y\}$ then according to (iii) the following are equivalent:
\begin{align*}
(c,d) & \in\Theta, \\
\text{there exists some }(e,f)\in\Theta\cap\{(x,y)\in P^2\mid x\le y\}\text{ with }(c,d) & \in[e,f], \\
\text{there exists some }(e,f)\in\Phi\cap\{(x,y)\in P^2\mid x\le y\}\text{ with }(c,d) & \in[e,f], \\
(c,d) & \in\Phi.
\end{align*}
\end{enumerate}
\end{proof}

It is worth noticing that by (iv) of Theorem~\ref{th1}, every congruence on a poset is fully determined by its couples of comparable elements.

We are going to show that in a poset satisfying the ACC and the DCC, in particular in every finite poset, every congruence class is a closed interval.

\begin{proposition}\label{prop2}
Let $\mathbf P=(P,\le)$ be poset satisfying the {\rm ACC} and the {\rm DCC}, $a\in P$ and $\Theta\in\Con\mathbf P$. Then $[a]\Theta$ is a closed interval.
\end{proposition}

\begin{proof}
Since $\mathbf P$ satisfies the DCC and the ACC, $[a]\Theta$ has some minimal element $b$ and some maximal element $c$. Since $[a]\Theta$ is convex we have $[b,c]\subseteq[a]\Theta$. Now let $d\in[a]\Theta$. Because of Theorem~\ref{th1} (i) there exists some $e\in\Max L(b,d)$ and some $f\in\Min U(b,d)$ with $e,f\in[a]\Theta$. Since $e\in[a]\Theta$, $e\le b$ and $b$ is a minimal element of $[a]\Theta$ we conclude $b=e\le d$. This shows that $b$ is the smallest element of $[a]\Theta$. Dually we obtain that $c$ is the greatest element of $[a]\Theta$ and hence $[a]\Theta\subseteq[b,c]$. Altogether, this yields $[a]\Theta=[b,c]$.
\end{proof}

For a congruence $\Theta$ on a poset with top element $1$, the class $[1]\Theta$ will be called the {\em kernel} of $\Theta$.

\begin{example}\label{ex3}
The poset $\mathbf P=(P,\le)$ depicted in Fig.~1
	
\vspace*{-4mm}
	
\begin{center}
\setlength{\unitlength}{7mm}
\begin{picture}(4,8)
\put(2,1){\circle*{.3}}
\put(1,3){\circle*{.3}}
\put(3,3){\circle*{.3}}
\put(1,5){\circle*{.3}}
\put(3,5){\circle*{.3}}
\put(2,7){\circle*{.3}}
\put(1,3){\line(0,1)2}
\put(1,3){\line(1,1)2}
\put(1,3){\line(1,-2)1}
\put(3,3){\line(-1,-2)1}
\put(3,3){\line(-1,1)2}
\put(3,3){\line(0,1)2}
\put(2,7){\line(-1,-2)1}
\put(2,7){\line(1,-2)1}
\put(1.85,.3){$0$}
\put(.35,2.85){$a$}
\put(3.35,2.85){$b$}
\put(.35,4.85){$c$}
\put(3.35,4.85){$d$}
\put(1.85,7.4){$1$}
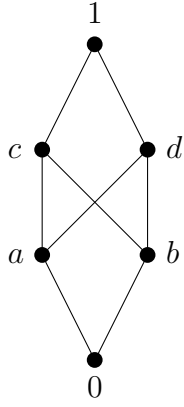
\put(-1.65,-.75){{\rm Figure~1. Non-lattice poset $\mathbf P$}}
\end{picture}
\end{center}
	
\vspace*{4mm}
	
has exactly eight non-trivial congruences:
\begin{align*}
\Theta_1 & :=[0,a]^2\cup[b,c]^2\cup[d,1]^2, \\
\Theta_2 & :=[0,a]^2\cup[b,d]^2\cup[c,1]^2, \\
\Theta_3 & :=[0,b]^2\cup[a,c]^2\cup[d,1]^2, \\
\Theta_4 & :=[0,b]^2\cup[a,d]^2\cup[c,1]^2, \\
\Theta_5 & :=[0,a]^2\cup[b,1]^2, \\ 
\Theta_6 & :=[0,c]^2\cup[d,1]^2, \\ 
\Theta_7 & :=[0,d]^2\cup[c,1]^2, \\
\Theta_8 & :=[0,b]^2\cup[a,1]^2.
\end{align*}
The Hasse diagram of $(\Con\mathbf P,\subseteq)$ is visualized in Fig.~2.

\vspace*{-4mm}

\begin{center}
\setlength{\unitlength}{7mm}
\begin{picture}(8,8)
\put(4,1){\circle*{.3}}
\put(1,3){\circle*{.3}}
\put(3,3){\circle*{.3}}
\put(5,3){\circle*{.3}}
\put(7,3){\circle*{.3}}
\put(1,5){\circle*{.3}}
\put(3,5){\circle*{.3}}
\put(5,5){\circle*{.3}}
\put(7,5){\circle*{.3}}
\put(4,7){\circle*{.3}}
\put(4,1){\line(-3,2)3}
\put(4,1){\line(-1,2)1}
\put(4,1){\line(1,2)1}
\put(4,1){\line(3,2)3}
\put(4,7){\line(-3,-2)3}
\put(4,7){\line(-1,-2)1}
\put(4,7){\line(1,-2)1}
\put(4,7){\line(3,-2)3}
\put(1,3){\line(0,1)2}
\put(1,3){\line(1,1)2}
\put(3,3){\line(-1,1)2}
\put(3,3){\line(1,1)2}
\put(5,3){\line(-1,1)2}
\put(5,3){\line(1,1)2}
\put(7,3){\line(-1,1)2}
\put(7,3){\line(0,1)2}
\put(3.75,.3){$\Delta$}
\put(.05,2.8){$\Theta_1$}
\put(2.05,2.8){$\Theta_2$}
\put(5.3,2.8){$\Theta_3$}
\put(7.3,2.8){$\Theta_4$}
\put(.05,4.8){$\Theta_5$}
\put(2.05,4.8){$\Theta_6$}
\put(5.3,4.8){$\Theta_7$}
\put(7.3,4.8){$\Theta_8$}
\put(3.75,7.4){$\nabla$}
\put(-.35,-.75){{\rm Figure~2. Congruence lattice of $\mathbf P$}}
\end{picture}
\end{center}

\vspace*{4mm}

Observe that $(\Con\mathbf P,\subseteq)$ is a lattice, but in this lattice the meet-operation does not coincide with set-theoretical intersection, namely
\[
\Theta_5\wedge\Theta_8=\Delta\ne\Delta\cup\{c,d,1\}^2=\Theta_5\cap\Theta_8.
\]
\end{example}

We can show that the smallest elements of the congruence classes $[a]\Theta$ and $[b]\Theta$ are ordered in the same way as their greatest elements provided these classes are closed intervals in $(P,\le)$.

\begin{proposition}\label{prop1}
Let $\mathbf P=(P,\le)$ be poset and $\Theta\in\Con\mathbf P$. Assume that the intervals $[a,b]$, $[c,d]$ are classes of $\Theta$. Then $a\le c$ if and only if $b\le d$.
\end{proposition}

\begin{proof}
First assume $a\le c$. Since $\Theta$ is compatible with the binary operator $\Min U$ we have
\[
\big(c\times\Min U(b,d)\big)\cap\Theta=\big(\Min U(a,c)\times\Min U(b,d)\big)\cap\Theta\ne\emptyset
\]
and hence there exists some $e\in\Min U(b,d)$ with $(c,e)\in\Theta$. Since $d\le e$ and $e\in[c]\Theta=[c,d]$ we have $b\le e=d$. That $b\le d$ implies $a\le c$ can be proved dually by using the compatibility of $\Theta$ with $\Max L$. 	
\end{proof}

Using the Propositions~\ref{prop2} and \ref{prop1} we show that the quotient set $P/\Theta$ can be converted into a poset that can be embedded into $\mathbf P$.

\begin{corollary}
Let $\mathbf P=(P,\le)$ be poset satisfying the {\rm ACC} and the {\rm DCC} and $\Theta\in\Con\mathbf P$. For all $A,B\in P/\Theta$ we define $A\le B$ if $\bigwedge A\le\bigwedge B$  Then $A\mapsto\bigwedge A$ is an order isomorphism from $(P/\Theta,\le)$ to $(\{\bigwedge A\mid A\in P/\Theta\},\le)$ and hence $(P/\Theta,\le)$ is a poset that can be embedded into $\mathbf P$.
\end{corollary}

An interesting connection between classes of congruences in posets is the following one. We will use it in the proof of Corollary~\ref{cor1}.

\begin{lemma}\label{lem5}
Let $\mathbf P=(P,\le)$ be a poset, $b\in P$, $\Theta\in\Con\mathbf P$ and $(a,c)\in\Theta$. Then the following holds:
\begin{enumerate}[{\rm(i)}]
\item If $[b]\Theta\cap U(c)\ne\emptyset$ then $\Max L(a,b)\cap[c]\Theta\ne\emptyset$,
\item if $[b]\Theta\cap L(c)\ne\emptyset$ then $\Min U(a,b)\cap[c]\Theta\ne\emptyset$.
\end{enumerate}
\end{lemma}

\begin{proof}
\
\begin{enumerate}[(i)]
\item If $d\in [b]\Theta\cap U(c)$ then because of $(a,c),(b,d)\in\Theta$ and $\Max L(c,d)=c$ there exists some $e\in\Max L(a,b)$ with $(e,c)\in\Theta$ and hence $e\in \Max L(a,b)\cap[c]\Theta$.
\item If $d\in [b]\Theta\cap L(c)$ then because of $(a,c),(b,d)\in\Theta$ and $\Min U(c,d)=c$ there exists some $e\in\Min U(a,b)$ with $(e,c)\in\Theta$ and hence $e\in \Min U(a,b)\cap[c]\Theta$.
\end{enumerate}
\end{proof}

Filters in posets are defined in various ways by different authors. We prefer the following definition.

\begin{definition}\label{def2}
A subset $F$ of a poset $\mathbf P=(P,\le)$ is called a {\em filter} of $\mathbf P$ if it satisfies the following conditions:
\begin{enumerate}[{\rm(i)}]
\item $F\ne\emptyset$,
\item $x\in F$, $y\in P$ and $x\le y$ imply $y\in F$.
\end{enumerate}
$F$ is called a {\em strong filter} of $\mathbf P$ if, moreover,
\begin{enumerate}
\item[{\rm(iii)}] $L(x,y)\cap F\ne\emptyset$ for all $x,y\in F$.
\end{enumerate}
\end{definition}

Now we prove the expected fact that the kernel of a congruence on a poset with top element $1$ is a strong filter.

\begin{corollary}\label{cor1}
Let $\mathbf P=(P,\le,1)$ be a poset with top element $1$, $a,b\in P$ and $\Theta\in\Con(P,\le)$ and put $F:=[1]\Theta$. Then the following holds:
\begin{enumerate}[{\rm(i)}]
\item If $a,b\in F$ then $\Max L(a,b)\cap F\ne\emptyset$,
\item if $b\in F$ then $\Max L(a,b)\cap[a]\Theta\ne\emptyset$,
\item if $a\in F$ then $\Min U(a,b)\cap F\ne\emptyset$,
\item $F$ is a strong filter of $\mathbf P$.
\end{enumerate}
\end{corollary}

\begin{proof}
\
\begin{enumerate}[(i)]
\item This follows from (i) of Lemma~\ref{lem5} by assuming $a,b\in F$ and putting $c=1$.
\item This follows from (i) of Lemma~\ref{lem5} by assuming $b\in F$ and putting $c=a$.
\item This follows from (ii) of Lemma~\ref{lem5} by assuming $a\in F$ and putting $c=1$.
\item Since $1\in F$ we have $F\ne\emptyset$. Condition (ii) of Definition~\ref{def2} follows from (ii) of Theorem~\ref{th1}. The rest follows from (i).
\end{enumerate}	
\end{proof}

\section{Relatively pseudocomplemented posets}

Let $\mathbf P=(P,\le)$ be a poset. If for all $x,y\in P$ there exists a greatest element $z$ of $P$ satisfying $L(x,z)\le y$ then $\mathbf P$ is called {\em relatively pseudocomplemented} and $z$ is called the {\em pseudocomplement of $x$ with respect to $y$}. We write $z=x*y$ and denote relatively pseudocomplemented posets in the form $(P,\le,*)$.

\begin{example}\label{ex1}
Consider the poset from Example~\ref{ex3}. This poset is relatively pseudocomplemented and the table for the binary operation $*$ is as follows:
\[
\begin{array}{r|rrrrrr}
* & 0 & a & b & c & d & 1 \\
\hline
0 & 1 & 1 & 1 & 1 & 1 & 1 \\
a & b & 1 & b & 1 & 1 & 1 \\
b & a & a & 1 & 1 & 1 & 1 \\
c & 0 & a & b & 1 & d & 1 \\
d & 0 & a & b & c & 1 & 1 \\
1 & 0 & a & b & c & d & 1
\end{array}
\]	
\end{example}

In the following lemma we list several properties of relative pseudocomplementation in posets showing that the operator $*$ shares many properties with usual implication in intuitionistic logics.

\begin{lemma}\label{lem1}
Let $\mathbf P=(P,\le,*)$ be a relatively pseudocomplemented poset and $a,b,c\in P$. Then the following holds:
\begin{enumerate}[{\rm(i)}]
\item $a\le b*c$ if and only if $L(a,b)\le c$, and $a\le b*c$ if and only if $b\le a*c$.
\item $a*a$ is the greatest element $1$ of $\mathbf P$, and $1*a=a$,
\item $a*b=1$ if and only if $a\le b$,
\item $b\le a*b$,
\item $a\le(a*b)*b$,
\item $a\le b$ implies $c*a\le c*b$ and $b*c\le a*c$,
\item $\big((a*b)*b\big)*b=a*b$,
\item $a*b\le\big((a*b)*a\big)*b$ if and only if $(a*b)*a\le(a*b)*b$,
\item $a*b\le\big((a*b)*b\big)*a$ if and only if $(a*b)*b\le(a*b)*a$,
\item $L(a,a*b)=L(a,b)$.
\end{enumerate}
\end{lemma}

\begin{proof}
\
\begin{enumerate}[(i)]
\item The first statement follows directly from the definition of $*$, and the second follows from the first.
\item follows directly from the definition of $*$.
\item According to (i) the following are equivalent: $a*b=1$, $1\le a*b$, $L(1,a)\le b$, $a\le b$.
\item This follows from $L(a,b)\le b$.
\item This follows from (i).
\item Assume $a\le b$. Then any of the following statements implies the next one: $c*a\le c*a$, $L(c*a,c)\le a$, $L(c*a,c)\le b$, $c*a\le c*b$. Moreover, any of the following statements implies the next one: $b*c\le b*c$, $L(b*c,b)\le c$, $L(b*c,a)\le c$, $b*c\le a*c$.
\item From (v) we obtain $\big((a*b)*b\big)*b\le a*b$ by (vi). On the other hand we have $a*b\le\big((a*b)*b\big)*b$ again according to (v).
\item This follows from (i).
\item This follows from (i).
\item According to (iv) we have $b\le a*b$ and hence $L(a,b)\subseteq L(a,a*b)$. Conversely, by the definition of relative pseudocomplementation we derive $L(a,a*b)\le b$ whence $L(a,a*b)\subseteq L(a,b)$.
\end{enumerate}	
\end{proof}

The concept of a congruence can be extended to pseudocomplemented posets as follows.

\begin{definition}\label{def1}
Let $\mathbf P=(P,\le,*)$ be a relatively pseudocomplemented poset and $\Theta$ a binary relation on $P$. Then we will call $\Theta$ a {\em congruence} on $\mathbf P$ if $\Theta$ is a congruence on $(P,\le)$ being compatible with $*$. Let $\Con\mathbf P$ denote the set of all congruences on $\mathbf P$.
\end{definition}

\begin{example}\label{ex2}
Consider again the poset $\mathbf P$ from Fig.~1 but now with the relative pseudocomplementation $*$. It is interesting to see how restrictive is compatibility with $*$. According to Example~\ref{ex3}, $\mathbf P$ has exactly ten different congruences, but only four of them are compatible with $*$, namely $\Delta$, $\Theta_5$, $\Theta_8$ and $\nabla$.
\end{example}

It is well known that in congruence permutable varieties of algebras there exists a so-called Malcev term, i.e.\ a ternary term $t(x,y,z)$ satisfying the identities
\[
t(x,x,z)\approx z\text{ and }t(x,z,z)\approx x.
\]
Of course, we cannot expect that such a term exists also in relatively pseudocomplemented posets but, surprisingly, we are able to derive a ternary operator sharing similar properties.

Let $P$ be a non-empty set. We call a ternary operator $T(x,y,z)$ on $P$ a {\em Malcev operator} if it satisfies the identities
\[
T(x,x,z)\approx z\text{ and }T(x,z,z)\approx x.
\]

\begin{lemma}\label{lem4}
Let $(P,\le,*)$ be a relatively pseudocomplemented poset satisfying the {\rm ACC} and define
\begin{enumerate}[{\rm(1)}]
\item $T(x,y,z):=\Max L\big((x*y)*z,(z*y)*x\big)$ for all $x,y,z\in P$.
\end{enumerate}
Then $T(x,y,z)$ is a Malcev operator.
\end{lemma}

\begin{proof}
For $a,c\in P$ according to Lemma~\ref{lem1} (i), (ii) and (iv) we have
\begin{align*}
T(a,a,c) & =\Max L\big((a*a)*c,(c*a)*a\big)=\Max L\big(1*c,(c*a)*a\big)= \\
         & =\Max L\big(c,(c*a)*a\big)=\Max L(c)=c, \\
T(a,c,c) & =\Max L\big((a*c)*c,(c*c)*a\big)=\Max L\big((a*c)*c,1*a\big)= \\
         & =\Max L\big((a*c)*c,a\big)=\Max L(a)=a.
\end{align*}
\end{proof}

Since the operator $T(x,y,z)$ is composed by means of the binary operators $\Max L$ and $*$, a binary relation compatible with $\Max L$ and $*$ is also compatible with $T(x,y,z)$.

Applying the Malcev operator from Lemma~\ref{lem4} we can prove the following result.

\begin{theorem}\label{th3}
Let $\mathbf P=(P,\le,*)$ be a relatively pseudocomplemented poset satisfying the {\rm ACC}, $T(x,y,z)$ the Malcev operator on $P$ defined by {\rm(1)} and $\Theta$ a reflexive binary relation on $P$ being compatible with the binary operators $\Max L$, $\Min U$ and $*$. Then $\Theta\in\Con\mathbf P$.
\end{theorem}

\begin{proof}
If $(a,b),(b,c)\in R$ then
\begin{align*}
\{(b,a)\}\cap\Theta & =\big(T(a,a,b)\times T(a,b,b)\big)\cap\Theta\ne\emptyset, \\
\{(a,c)\}\cap\Theta & =\big(T(a,b,b)\times T(b,b,c)\big)\cap\Theta\ne\emptyset
\end{align*}
and hence $(b,a),(a,c)\in\Theta$ showing symmetry and transitivity of $\Theta$. The remaining part of the proof follows from the definition of a congruence on $\mathbf P$.
\end{proof}

\begin{example}
Consider the relatively pseudocomplemented poset $\mathbf P=(P,\le,*)$ depicted in Fig.~3.
	
\vspace*{-4mm}
	
\begin{center}
\setlength{\unitlength}{7mm}
\begin{picture}(4,10)
\put(2,1){\circle*{.3}}
\put(1,3){\circle*{.3}}
\put(3,3){\circle*{.3}}
\put(1,5){\circle*{.3}}
\put(3,5){\circle*{.3}}
\put(1,7){\circle*{.3}}
\put(3,7){\circle*{.3}}
\put(2,9){\circle*{.3}}
\put(1,3){\line(0,1)4}
\put(1,3){\line(1,1)2}
\put(1,3){\line(1,-2)1}
\put(3,3){\line(-1,-2)1}
\put(3,3){\line(-1,1)2}
\put(3,3){\line(0,1)4}
\put(2,9){\line(-1,-2)1}
\put(2,9){\line(1,-2)1}
\put(1,5){\line(1,1)2}
\put(3,5){\line(-1,1)2}
\put(1.85,.3){$0$}
\put(.35,2.85){$a$}
\put(3.35,2.85){$b$}
\put(.35,4.85){$c$}
\put(3.35,4.85){$d$}
\put(.35,6.85){$e$}
\put(3.35,6.85){$f$}
\put(1.85,9.4){$1$}
\put(-4.1,-.75){{\rm Figure~3. Relatively pseudocomplemented poset}}
\end{picture}
\end{center}
	
\vspace*{4mm}
	
The table for the binary operation $*$ is as follows:
\[
\begin{array}{r|rrrrrrrr}
* & 0 & a & b & c & d & e & f & 1 \\
\hline
0 & 1 & 1 & 1 & 1 & 1 & 1 & 1 & 1 \\
a & b & 1 & b & 1 & 1 & 1 & 1 & 1 \\
b & a & a & 1 & 1 & 1 & 1 & 1 & 1 \\
c & 0 & a & b & 1 & d & 1 & 1 & 1 \\
d & 0 & a & b & c & 1 & 1 & 1 & 1 \\
e & 0 & a & b & c & d & 1 & f & 1 \\
f & 0 & a & b & c & d & e & 1 & 1 \\
1 & 0 & a & b & c & d & e & f & 1
\end{array}
\]
The relatively pseudocomplemented poset $\mathbf P$ has only the following two non-trivial congruences:
\begin{align*}
\Theta_1 & :=[0,a]^2\cup[b,1]^2, \\
\Theta_2 & :=[0,b]^2\cup[a,1]^2.
\end{align*}
Observe that $(\Con\mathbf P,\subseteq)$ is a lattice, but again
\[
\Theta_1\wedge\Theta_2=\Delta\ne\Delta\cup\{c,d,e,f,1\}^2=\Theta_1\cap\Theta_2.
\]
\end{example}

The next lemma shows some connections between a congruence on a relatively pseudocomplemented poset and its kernel.

\begin{lemma}\label{lem3}
Let $\mathbf P=(P,\le,*)$ be a relatively pseudocomplemented poset, $a,b\in P$ and $\Theta\in\Con(P,*)$. Then the following holds:
\begin{enumerate}[{\rm(i)}]
\item If $(a,b)\in\Theta$ then $a*b,b*a\in[1]\Theta$,
\item if $a*b,b*a\in[1]\Theta$ and $\big((a*b)*b,(b*a)*a\big)\in\Theta$ then $(a,b)\in\Theta$.
\end{enumerate}
\end{lemma}

\begin{proof}
\
\begin{enumerate}[(i)]
\item If $(a,b)\in\Theta$ then $a*b,b*a\in[a*a]\Theta=[1]\Theta$.
\item If $a*b,b*a\in[1]\Theta$ and $\big((a*b)*b,(b*a)*a\big)\in\Theta$ then
\[
a=1*a\mathrel\Theta(b*a)*a\mathrel\Theta(a*b)*b\mathrel\Theta1*b=b.
\]
\end{enumerate}	
\end{proof}

We demonstrate the aforementioned result by means of Example~\ref{ex1}. Let $\mathbf P$ denote the poset from this example and $\Theta_1$ the congruence on $\mathbf P$ from Example~\ref{ex2}. Then $b*c=1\in[1]\Theta_1$, $c*b=b\in[1]\Theta_1$ and $\big((b*c)*c,(c*b)*b\big)=(1*c,b*b)=(c,1)\in\Theta_1$, thus $(b,c)\in\Theta_1$.

In Lemma~\ref{lem3} we described a certain connection between a congruence $\Theta$ on a relatively pseudocomplemented poset and its kernel $[1]\Theta$. In order to analyze this connection more deeply, we define the following concept.

\begin{definition}
A subset $D$ of a relatively pseudocomplemented poset $\mathbf P=(P,\le,*)$ is called a {\em deductive system} of $\mathbf P$ if it satisfies the following conditions:
\begin{enumerate}[{\rm(i)}]
\item $1\in D$,
\item $x\in D$, $y\in P$ and $x*y\in D$ imply $y\in D$.
\end{enumerate}
\end{definition}

Evidently, for each $\Theta\in\Con\mathbf P$ the set $[1]\Theta$ is a deductive system of $\mathbf P$. Namely, $1\in[1]\Theta$ and if $x\in[1]\Theta$, $y\in P$ and $x*y\in[1]\Theta$ then $y=1*y\in[x*y]\Theta=[1]\Theta$.

We are able to characterize deductive systems by means of so-called ideal terms. Recall that the term $p(x_1,\ldots,x_n,y_1,\ldots,y_k)$ is called an {\em ideal term} in $y_1,\ldots,y_k$ if
\[
p(a_1,\ldots,a_n,1,\ldots,1)=1\text{ for all }a_1,\ldots,a_n\in P.
\]
A {\em subset} $D$ of $P$ is called {\em closed} with respect to the ideal term $p(x_1,\ldots,x_n,y_1,\ldots,y_k)$ if $p(a_1,\ldots,a_n,b_1,\ldots,b_k)\in D$ for all $a_1,\ldots,a_n\in P$ and all $b_1,\ldots,b_k\in D$.

In our case, a deductive system is determined by two simple ideal terms.

\begin{theorem}
Let $\mathbf P=(P,\le,*)$ be a relatively pseudocomplemented poset, $a,b,c\in P$, $D\subseteq P$ and $\Theta\in\Con\mathbf P$ and define $t_1:=1$ and $t_2(x,y_1,y_2):=\big(y_1*(y_2*x)\big)*x$. Then the following holds:
\begin{enumerate}[{\rm(i)}]
\item $t_1$ and $t_2$ are ideal terms,
\item if $D$ is closed with respect to $t_1$ and $t_2$ then $D$ is a deductive system of $\mathbf P$,
\item $[1]\Theta$ is closed with respect to $t_1$ and $t_2$.
\end{enumerate}
\end{theorem}

\begin{proof}
\
\begin{enumerate}[(i)]
\item We have $t_1=1$ and $t_2(a,1,1)=\big(1*(1*a)\big)*a=(1*a)*a=a*a=1$.
\item We have $1=t_1\in D$ and if $a,a*b\in D$ then
\[
b=1*b=\big((a*b)*(a*b)\big)*b=t_2(b,a*b,a)\in D.
\]
\item We have $t_1=1\in[1]\Theta$ and if $b,c\in[1]\Theta$ then
\[
t_1(a,b,c)=\big(b*(c*a)\big)*a\in[\big(1*(1*a)\big)*a]\Theta=[(1*a)*a]\Theta=[a*a]\Theta=[1]\Theta.
\]
\end{enumerate} 
\end{proof}

For a subset $A$ of a relatively pseudocomplemented poset $(P,\le,*)$ we define a binary relation $\Theta_A$ on $P$ as follows:
\[
(x,y)\in\Theta_A\text{ if there exist }a,b\in A\text{ with }L(x,a,b)=L(y,a,b).
\]

\begin{definition}
A subset $F$ of a relatively pseudocomplemented poset $\mathbf P=(P,\le,*)$ is called a {\em filter} of $\mathbf P$ if it is a filter of $(P,\le)$ satisfying $x*y\in F$ for all $x,y\in F$. A {\em filter} of $\mathbf P$ is called {\em strong} if it is a strong filter of $(P,\le)$.
\end{definition}

\begin{lemma}
Let $\mathbf P=(P,\le,*)$ be a relatively pseudocomplemented poset, $D$ a deductive system of $\mathbf P$ and $\Theta\in\Con\mathbf P$. Then $D$ is a filter of $\mathbf P$ and $[1]\Theta$ a strong filter of $\mathbf P$.
\end{lemma}

\begin{proof}
Because of $1\in D$ we have $D\ne\emptyset$. If $a\in D$, $b\in P$ and $a\le b$ then $a*b=1\in D$ by (ii) of Lemma~\ref{lem1} and hence $b\in D$. If, finally, $a,b\in D$ then because of $b\le a*b$ (according to (iv) of Lemma~\ref{lem1}) we have $a*b\in D$ showing that $D$ is a filter of $\mathbf P$. According to Corollary~\ref{cor1} $[1]\Theta$ is a strong filter of $(P,\le)$. If $a,b\in[1]\Theta$ then $a*b\in[1*1]\Theta=[1]\Theta$ by (iii) of Lemma~\ref{lem1} showing that $[1]\Theta$ is a strong filter of $\mathbf P$.
\end{proof}

Though, according to Lemma~\ref{lem3}, $\Theta$ seems in general not to be determined by $[1]\Theta$, we can show that in case of a strong filter $F$, $\Theta_F$ is determined by $F$.

\begin{proposition}
Let $\mathbf P=(P,\le,*)$ be a relatively pseudocomplemented poset, $a,b\in P$ and $F$ a strong filter of $\mathbf P$. Then $(a,b)\in\Theta_F$ if and only if $a*b,b*a\in F$.
\end{proposition}

\begin{proof}
Assume $(a,b)\in\Theta_F$. Then there exist $c,d\in F$ with $L(a,c,d)=L(b,c,d)$. Because of $L(a,c,d)\le b$ and $L(b,c,d)\le a$ we have $L(c,d)\le a*b$ and $L(c,d)\le b*a$, respectively. Since $F$ is compatible we conclude $L(c,d)\cap F\ne\emptyset$ and hence $a*b,b*a\in F$. If, conversely, $a*b,b*a\in F$ then because of (x) of Lemma~\ref{lem1} we have
\begin{align*}
L(a,a*b,b*a) & =L(a,b,b*a)=L(a,b,a)=L(b,a,b)=L(b,a,a*b)= \\
             & =L(b,b*a,a*b)=L(b,a*b,b*a)
\end{align*}
and hence $(a,b)\in\Theta_F$.
\end{proof}

\section{Boolean posets}

Now we turn our attention to another sort of posets generalizing Boolean algebras, namely the so-called Boolean posets.

Recall that a {\em poset} $(P,\le)$ is called {\em distributive} if it satisfies one of the following equivalent identities:
\begin{align*}
 L\big(U(x,y),z\big) & \approx LU\big(L(x,z),L(y,z)\big), \\	
UL\big(U(x,y),z\big) & \approx U\big(L(x,z),L(y,z)\big), \\	
 U\big(L(x,y),z\big) & \approx UL\big(U(x,z),U(y,z)\big), \\
LU\big(L(x,y),z\big) & \approx L\big(U(x,z),U(y,z)\big).
\end{align*}
A unary operation $'$ on a bounded poset $\mathbf P=(P,\le,0,1)$ is called a {\em complementation} on $\mathbf P$ if $\mathbf P$ satisfies the identities
\[
U(x,x')\approx1\text{ and }L(x,x')\approx0.
\]
In this case, $(P,\le,{}',0,1)$ is called a {\em complemented poset}. A distributive complemented poset is called a {\em Boolean poset}.

It is well-known that a distributive poset can have at most one complementation.

\begin{definition}
Let $\mathbf P=(P,\le,{}',0,1)$ be a Boolean poset and $\Theta$ a binary relation on $P$. Then we will call $\Theta$ a {\em congruence} on $\mathbf P$ if $\Theta$ is a congruence on $(P,\le)$ being compatible with $'$. Let $\Con\mathbf P$ denote the set of all congruences on $\mathbf P$.
\end{definition}

\begin{example}\label{ex5}
The Boolean poset $\mathbf P=(P,\le,{}',0,1)$ visualizeded in Fig.~4

\vspace*{-4mm}

\begin{center}
\setlength{\unitlength}{7mm}
\begin{picture}(8,8)
\put(4,1){\circle*{.3}}
\put(1,3){\circle*{.3}}
\put(3,3){\circle*{.3}}
\put(5,3){\circle*{.3}}
\put(7,3){\circle*{.3}}
\put(1,5){\circle*{.3}}
\put(3,5){\circle*{.3}}
\put(5,5){\circle*{.3}}
\put(7,5){\circle*{.3}}
\put(4,7){\circle*{.3}}
\put(4,1){\line(-3,2)3}
\put(4,1){\line(-1,2)1}
\put(4,1){\line(1,2)1}
\put(4,1){\line(3,2)3}
\put(4,7){\line(-3,-2)3}
\put(4,7){\line(-1,-2)1}
\put(4,7){\line(1,-2)1}
\put(4,7){\line(3,-2)3}
\put(1,3){\line(0,1)2}
\put(1,3){\line(1,1)2}
\put(1,3){\line(2,1)4}
\put(3,3){\line(-1,1)2}
\put(3,3){\line(0,1)2}
\put(3,3){\line(2,1)4}
\put(5,3){\line(-2,1)4}
\put(5,3){\line(0,1)2}
\put(5,3){\line(1,1)2}
\put(7,3){\line(-2,1)4}
\put(7,3){\line(-1,1)2}
\put(7,3){\line(0,1)2}
\put(3.85,.3){$0$}
\put(.35,2.85){$a$}
\put(2.35,2.85){$b$}
\put(5.4,2.85){$c$}
\put(7.4,2.85){$d$}
\put(.35,4.85){$d'$}
\put(2.35,4.85){$c'$}
\put(5.4,4.85){$b'$}
\put(7.4,4.85){$a'$}
\put(3.85,7.4){$1$}
\put(.7,-.75){{\rm Figure~4. Boolean poset $\mathbf P$}}
\end{picture}
\end{center}

\vspace*{4mm}

has exactly four non-trivial congruences, namely
\begin{align*}
\Theta_1 & :=[0,a']^2\cup[a,1]^2, \\
\Theta_2 & :=[0,b']^2\cup[b,1]^2, \\
\Theta_3 & :=[0,c']^2\cup[c,1]^2, \\
\Theta_4 & :=[0,d']^2\cup[d,1]^2.
\end{align*}
The Hasse diagram of $(\Con\mathbf P,\subseteq)$ is depicted in Fig.~5.

\vspace*{-4mm}

\begin{center}
\setlength{\unitlength}{7mm}
\begin{picture}(8,6)
\put(4,1){\circle*{.3}}
\put(1,3){\circle*{.3}}
\put(3,3){\circle*{.3}}
\put(5,3){\circle*{.3}}
\put(7,3){\circle*{.3}}
\put(4,5){\circle*{.3}}
\put(4,1){\line(-3,2)3}
\put(4,1){\line(-1,2)1}
\put(4,1){\line(1,2)1}
\put(4,1){\line(3,2)3}
\put(4,5){\line(-3,-2)3}
\put(4,5){\line(-1,-2)1}
\put(4,5){\line(1,-2)1}
\put(4,5){\line(3,-2)3}
\put(3.75,.3){$\Delta$}
\put(.05,2.8){$\Theta_1$}
\put(2.05,2.8){$\Theta_2$}
\put(5.3,2.8){$\Theta_3$}
\put(7.3,2.8){$\Theta_4$}
\put(3.75,5.4){$\nabla$}
\put(-.35,-.75){{\rm Figure~5. Congruence lattice of $\mathbf P$}}
\end{picture}
\end{center}

\vspace*{4mm}

Observe that $(\Con\mathbf P,\subseteq)$ is a lattice, but in this lattice the meet-operation does not coincide with set-theoretical intersection, namely
\[
\Theta_1\wedge\Theta_2=\Delta\ne\{0,c,d\}^2\cup[a,b']^2\cup[b,a']^2\cup\{c',d',1\}^2=\Theta_1\cap\Theta_2.
\]
\end{example}

\begin{example}\label{ex4}
The Boolean poset $\mathbf P=(P,\le,{}',0,1)$ visualized in Fig.~6
	
\vspace*{-4mm}
	
\begin{center}
\setlength{\unitlength}{7mm}
\begin{picture}(8,10)
\put(4,1){\circle*{.3}}
\put(1,3){\circle*{.3}}
\put(3,3){\circle*{.3}}
\put(5,3){\circle*{.3}}
\put(7,3){\circle*{.3}}
\put(1,7){\circle*{.3}}
\put(3,7){\circle*{.3}}
\put(5,7){\circle*{.3}}
\put(7,7){\circle*{.3}}
\put(4,9){\circle*{.3}}
\put(1,5){\circle*{.3}}
\put(7,5){\circle*{.3}}
\put(4,1){\line(-3,2)3}
\put(4,1){\line(-1,2)1}
\put(4,1){\line(1,2)1}
\put(4,1){\line(3,2)3}
\put(4,9){\line(-3,-2)3}
\put(4,9){\line(-1,-2)1}
\put(4,9){\line(1,-2)1}
\put(4,9){\line(3,-2)3}
\put(1,3){\line(0,1)4}
\put(1,3){\line(1,1)4}
\put(3,3){\line(-1,1)2}
\put(3,3){\line(1,1)4}
\put(5,3){\line(-1,1)4}
\put(5,3){\line(1,1)2}
\put(7,3){\line(-1,1)4}
\put(7,3){\line(0,1)4}
\put(1,5){\line(1,1)2}
\put(7,5){\line(-1,1)2}
\put(3.85,.3){$0$}
\put(.35,2.85){$a$}
\put(2.35,2.85){$b$}
\put(5.4,2.85){$c$}
\put(7.4,2.85){$d$}
\put(.35,4.85){$e$}
\put(7.4,4.85){$e'$}
\put(.35,6.85){$d'$}
\put(2.35,6.85){$c'$}
\put(5.4,6.85){$b'$}
\put(7.4,6.85){$a'$}
\put(3.85,9.4){$1$}
\put(.45,-.75){{\rm Figure~6. Boolean poset $\mathbf P$}}
\end{picture}
\end{center}
	
\vspace*{4mm}
	
has exactly eight non-trivial congruences, namely
\begin{align*}
\Theta_1 & :=[0,a]^2\cup[b,e]^2\cup[c,d']^2\cup[d,c']^2\cup[e',b']^2\cup[a',1]^2, \\
\Theta_2 & :=[0,d]^2\cup[c,e']^2\cup[b,a']^2\cup[a,b']^2\cup[e,c']^2\cup[d',1]^2, \\
\Theta_3 & :=[0,e]^2\cup[c,d']^2\cup[d,c']^2\cup[e',1]^2, \\
\Theta_4 & :=[0,e']^2\cup[b,a']^2\cup[a,b']^2\cup[e,1]^2, \\
\Theta_5 & :=[0,c']^2\cup[c,1]^2, \\
\Theta_6 & :=[0,d']^2\cup[d,1]^2, \\
\Theta_7 & :=[0,a']^2\cup[a,1]^2, \\
\Theta_6 & :=[0,b']^2\cup[b,1]^2.
\end{align*}
The Hasse diagram of $(\Con\mathbf P,\subseteq)$ is depicted in Fig.~7.
	
\vspace*{-4mm}
	
\begin{center}
\setlength{\unitlength}{7mm}
\begin{picture}(8,10)
\put(4,1){\circle*{.3}}
\put(3,3){\circle*{.3}}
\put(5,3){\circle*{.3}}
\put(3,5){\circle*{.3}}
\put(5,5){\circle*{.3}}
\put(1,7){\circle*{.3}}
\put(3,7){\circle*{.3}}
\put(5,7){\circle*{.3}}
\put(7,7){\circle*{.3}}
\put(4,9){\circle*{.3}}
\put(4,1){\line(-1,2)1}
\put(4,1){\line(1,2)1}
\put(3,3){\line(0,1)4}
\put(5,3){\line(0,1)4}
\put(3,5){\line(-1,1)2}
\put(5,5){\line(1,1)2}
\put(4,9){\line(-3,-2)3}
\put(4,9){\line(-1,-2)1}
\put(4,9){\line(1,-2)1}
\put(4,9){\line(3,-2)3}
\put(3.75,.3){$\Delta$}
\put(.05,6.8){$\Theta_5$}
\put(2.05,6.8){$\Theta_6$}
\put(5.3,6.8){$\Theta_7$}
\put(7.3,6.8){$\Theta_8$}
\put(2.05,4.8){$\Theta_3$}
\put(5.3,4.8){$\Theta_4$}
\put(2.05,2.8){$\Theta_1$}
\put(5.3,2.8){$\Theta_2$}
\put(3.75,9.4){$\nabla$}
\put(-.35,-.75){{\rm Figure~7. Congruence lattice of $\mathbf P$}}
\end{picture}
\end{center}
	
\vspace*{4mm}
	
\end{example}

One can see that in the Boolean poset from Example~\ref{ex5} the filters $[a',1]$, $[b',1]$, $[c',1]$ and $[d',1]$ are not congruence kernels and that the same is true for the filters $[b',1]$ and $[c',1]$ of the Boolean poset from Example~\ref{ex4}. The reason is as follows.

\begin{lemma}\label{lem2}
Let $\mathbf P=(P,\le,{}',0,1)$ be a Boolean poset and $a\in P$. Then the following holds:
\begin{enumerate}[{\rm(i)}]
\item If there exists some $b\in P$ such that for all $c\in\Max L(a,b)$ there exists some $d\in U(c)\setminus[a,1]$ with $b\vee d=1$ then $[a,1]$ is not a congruence kernel of $\mathbf P$.
\item If there exists some $b\in P$ such that for all $c\in\Min U(a',b)$ there exists some $d\in U(b)\setminus[a,1]$ with $c\vee d=1$ then $[a,1]$ is not a congruence kernel of $\mathbf P$.
\end{enumerate}
\end{lemma}

\begin{proof}
\
\begin{enumerate}[(i)]
\item Suppose there exists some $b\in P$ such that for all $c\in\Max L(a,b)$ there exists some $d\in U(c)\setminus[a,1]$ with $b\vee d=1$. Assume $[a,1]$ to be a congruence kernel of $\mathbf P$. Then there exists some $\Theta\in\Con\mathbf P$ satisfying $[1]\Theta=[a,1]$. Since $(a,1),(b,b)\in\Theta$ there exists some $c\in\Max L(a,b)$ with $(c,b)\in\Theta$. According to our assumption there exists some $d\in U(c)\setminus[a,1]$ with $b\vee d=1$. Since $(c,b),(d,d)\in\Theta$, $\Min U(c,d)=d$ and $\Min U(b,d)=1$ we have $(d,1)\in\Theta$, i.e.\ $d\in[1]\Theta=[a,1]$ contradicting $d\in U(c)\setminus[a,1]$. Hence $[a,1]$ is not a congruence kernel of $\mathbf P$.
\item Suppose there exists some $b\in P$ such that for all $c\in\Min U(a',b)$ there exists some $d\in U(b)\setminus[a,1]$ with $c\vee d=1$. Assume $[a,1]$ to be a congruence kernel of $\mathbf P$. Then there exists some $\Theta\in\Con\mathbf P$ satisfying $[1]\Theta=[a,1]$. Now $(a,1)\in\Theta$ implies $(0,a')\in\Theta$. Since $(0,a'),(b,b)\in\Theta$ there exists some $c\in\Min U(a',b)$ with $(b,c)\in\Theta$. According to our assumption there exists some $d\in U(b)\setminus[a,1]$ with $c\vee d=1$. Since $(b,c),(d,d)\in\Theta$, $\Min U(b,d)=d$ and $\Min U(c,d)=1$ we have $(d,1)\in\Theta$, i.e.\ $d\in[1]\Theta=[a,1]$ contradicting $d\in U(b)\setminus[a,1]$. Hence $[a,1]$ is not a congruence kernel of $\mathbf P$.
\end{enumerate}
\end{proof}

Observe that (i) of Lemma~\ref{lem2} holds also for posets $(P,\le,1)$ with top element $1$.

We can now apply Lemma~\ref{lem2} to the Boolean posets from Examples~\ref{ex5} and \ref{ex4}. Consider the filter $[a',1]$ of the Boolean poset from Example~\ref{ex5}. Then $\Max L(a',b')=\{c,d\}$. Now $d'\in U(c)\setminus[a',1]$, $b'\vee d'=1$, $c'\in U(d)\setminus[a',1]$ and $b'\vee c'=1$. According to (i) of Lemma~\ref{lem2}, $[a',1]$ is not a congruence kernel of $\mathbf P$. Next consider the filter $[b',1]$ of the Boolean poset from Example~\ref{ex4}. Then $\Min U(b,c)=\{a',d'\}$. Now $d',e'\in U(c)\setminus[b',1]$, $a'\vee d'=1$ and $d'\vee e'=1$. According to (ii) of Lemma~\ref{lem2}, $[b',1]$ is not a congruence kernel of $\mathbf P$.

It is well-known that in Boolean algebras there exists a so-called Pixley term, i.e.\ a ternary term $t(x,y,z)$ satisfying the identities $t(x,z,z)\approx x$, $t(x,y,x)\approx x$ and $t(x,x,z)\approx z$. Of course, every Pixley term is a Malcev term. Surprisingly, in Boolean posets $\mathbf P=(P,\le,{}',0,1)$ satisfying the ACC and the DCC a ternary operator $T(x,y,z)$ can be defined satisfying these identities in contrast to the fact that the values of $T$ on triples different from those mentioned above may be proper subsets of $P$. It was shown in \cite{CLa} that the ternary operator $T(x,y,z)$ on $P$ defined by
\[
T(x,y,z):=\Min U\big(\Max L(x,z),\Max L(x,y',z'),\Max L(x',y',z)\big)
\]
for all $x,y,z\in P$ is a Pixley operator. Since the operator $T(x,y,z)$ is composed by means of the binary operators $\Max L$ and $\Min U$ as well as the unary operation $'$, a binary relation compatible with $\Max L$, $\Min U$ and $'$ is also compatible with $T(x,y,z)$.

Applying the above Pixley operator we can prove the following result.

The proof of the next theorem follows the same lines as that of Theorem~\ref{th3} and is therefore omitted.

\begin{theorem}
If $\mathbf P=(P,\le,{}',0,1)$ is a Boolean poset satisfying the {\rm ACC} and the {\rm DCC}, $T(x,y,z)$ the above defined Pixley operator on $P$ and $\Theta$ a reflexive binary relation on $P$ being compatible with the binary operators $\Max L$ and $\Min U$ and the unary operation $'$ then $\Theta\in\Con\mathbf P$.
\end{theorem}

\begin{remark}
Although Boolean algebras are congruence permutable, congruence regular and congruence uniform, the same does not hold for Boolean posets. Consider the Boolean poset $\mathbf P$ from Example~\ref{ex4}. Although Boolean posets possess a Pixley operator, $\mathbf P$ is not congruence permutable since $(0,a)\in\Theta_1$ and $(a,b')\in\Theta_2$ and hence $(0,b')\in\Theta_1\circ\Theta_2$, but there does not exist some $x\in P$ satisfying $(0,x)\in\Theta_2$ and $(x,b')\in\Theta_1$ and therefore $(0,b')\notin\Theta_2\circ\Theta_1$. Further, $\mathbf P$ is not congruence regular because e.g.\ the congruences $\Theta_1$ and $\Theta_3$ have the class $[d,c']$ in common, and the congruences $\Theta_2$ and $\Theta_4$ have the class $[b,a']$ in common. Moreover, $\mathbf P$ is not congruence uniform since the congruences $\Theta_3$ and $\Theta_4$ have classes of different cardinalities. Finally, not every principal filter of $\mathbf P$ is a kernel of some congruence on $\mathbf P$ as shown after the proof of Lemma~\ref{lem2}.
\end{remark}	

Now we show connections between a congruence $\Theta$ on a Boolean poset and its kernel. As mentioned in Theorem~\ref{th1}, we can restrict ourselves to those pairs in $\Theta$ which consist of comparable elements.

\begin{theorem}\label{th2}
Let $\mathbf P=(P,\le,{}',0,1)$ be a Boolean poset, $a,b\in P$ and $\Theta\in\Con\mathbf P$. Then the following holds:
\begin{enumerate}[{\rm(i)}]
\item If $(a,b)\in\Theta$ then $\Min U(a,b')\cap[1]\Theta\ne\emptyset$,
\item if $a\le b$ and $LU(a,b')\cap[1]\Theta\ne\emptyset$ then $(a,b)\in\Theta$,
\item if $a\le b$ and $a\vee b'$ is defined then $(a,b)\in\Theta$ if and only if $a\vee b'\in[1]\Theta$.
\end{enumerate}
\end{theorem}

\begin{proof}
\
\begin{enumerate}[(i)]
\item If $(a,b)\in\Theta$ then because of $(b',b')\in\Theta$ and $\Min U(b,b')=1$ there exists some $c\in\Min U(a,b')$ with $(c,1)\in\Theta$, i.e.\ $c\in\Min U(a,b')\cap[1]\Theta$ and hence $\Min U(a,b')\cap[1]\Theta\ne\emptyset$.
\item Since $LU(a,b')\cap[1]\Theta\ne\emptyset$ there exists some $d\in LU(a,b')\cap[1]\Theta$. Because $(d,1),(b,b)$ $\in\Theta$ and $\Max L(1,b)=b$ there exists some $e\in\Max L(d,b)$ with $(e,b)\in\Theta$. Now
\[
e\in L(d,b)\subseteq L\big(U(a,b'),b\big)=LU\big(L(a,b),L(b',b)\big)=LU\big(L(a),0\big)=L(a)
\]
and hence $e\le a$. Together we obtain $e\le a\le b$ which together with $(e,b)\in\Theta$ implies $(a,b)\in\Theta$ according to Theorem~\ref{th1} (ii).
\item Assume $a\le b$ and $a\vee b'$ is defined. Then $\Min U(a,b')=a\vee b'$ and $LU(a,b')=[0,a\vee b']$. Since $[1]\Theta$ is convex according to (ii) of Theorem~\ref{th1}, we have $LU(a,b')\cap[1]\Theta\ne\emptyset$ if and only if $a\vee b'\in[1]\Theta$. Now (iii) follows from (i) and (ii).
\end{enumerate}
\end{proof}

The following result shows that a certain kind of so-called weak congruence regularity holds also for Boolean posets.

\begin{corollary}\label{cor2}
Let $\mathbf P=(P,\le,{}',0,1)$ be a Boolean poset and $\Theta,\Phi\in\Con\mathbf P$. Further assume
\begin{enumerate}
\item[{\rm(2)}]
$LU(x,y')\cap[1]\Psi\ne\emptyset$ for all $\Psi\in\{\Theta,\Phi\}$ and all $(x,y)\in\Psi$ with $x\le y$.
\end{enumerate}
Moreover, suppose $[1]\Theta=[1]\Phi$. Then $\Theta=\Phi$.
\end{corollary}

\begin{proof}
Let $a,b\in P$ with $a\le b$ and $\Psi\in\{\Theta,\Phi\}$. Then, by (2) and Theorem~\ref{th2}, $(a,b)\in\Psi$ if and only if $LU(a,b')\cap[1]\Psi\ne\emptyset$. Since $[1]\Theta=[1]\Phi$ we obtain $(a,b)\in\Theta$ if and only if $(a,b)\in\Phi$. Now $\Theta=\Phi$ follows from (iv) of Theorem~\ref{th1}.
\end{proof}

\begin{remark}
We can show that assumption {\rm(2)} is not strange. For example if $\mathbf P$ denotes the Boolean poset from Example~\ref{ex4} then $LU(x,y')\cap[1]\Theta\ne\emptyset$ for all $\Theta\in\Con\mathbf P$ and all $(x,y)\in\Theta$ with $x\le y$. This statement easily follows by observing
\[
\{(x,y)\in P^2\mid x\le y,x\vee y'\text{ is not defined}\}=(\{a,b\}\times\{c',d'\})\cup(\{c,d\}\times\{a',b'\}).
\]
and {\rm(iii)} of Theorem~\ref{th2}.
\end{remark}

{\bf Data availability statement} No datasets were generated or analyzed during the current study.

\section{Declarations}

{\bf Funding} Support of the first author by the Czech Science Foundation (GA\v CR), project 25-20013L, and by IGA, project P\v rF~2025~008, is gratefully acknowledged. Support of the second author by the Austrian Science Fund (FWF), project 10.55776/PIN5424624, is gratefully acknowledged.

{\bf Availability of data and materials} No datasets were generated or analyzed during the current study.

{\bf Competing interests} There are no competing interests of a financial or personal nature between the authors.

{\bf Authors' contributions} Both authors contributed equally to the manuscript.

{\bf Compliance with Ethical Standards} This article does not contain any studies with human participants or animals performed by any of the authors.

{\bf Conflict of interest} Both authors declare that they have no conflict of interest.

%{\bf Declarations}

%{\bf Compliance with Ethical Standards} This article does not contain any studies with human participants or animals performed by any of the authors.

%{\bf Funding} This research was funded in part by the Austrian Science Fund (FWF) 10.55776/PIN5424624 and the Czech Science Foundation (GA\v CR) 25-20013L. Support of the research of the first author by IGA, project P\v rF~2025~008, is gratefully acknowledged.

%{\bf Data availability statement} No datasets were generated or analyzed during the current study.

%{\bf Competing interests} There are no competing interests of a financial or personal nature between the authors.

%{\bf Conflict of interest} Both authors declare that they have no conflict of interest.

%{\bf Authors' contributions} Both authors contributed equally to the manuscript.

Authors' addresses:

Ivan Chajda \\
Palack\'y University Olomouc \\
Faculty of Science \\
Department of Algebra and Geometry \\
17.\ listopadu 12 \\
771 46 Olomouc \\
Czech Republic \\
ivan.chajda@upol.cz

Helmut L\"anger \\
TU Wien \\
Faculty of Mathematics and Geoinformation \\
Institute of Discrete Mathematics and Geometry \\
Wiedner Hauptstra\ss e 8-10 \\
1040 Vienna \\
Austria, and \\
Palack\'y University Olomouc \\
Faculty of Science \\
Department of Algebra and Geometry \\
17.\ listopadu 12 \\
771 46 Olomouc \\
Czech Republic \\
helmut.laenger@tuwien.ac.at
\end{document}